\documentclass[11pt,reqno]{article}
\usepackage{amsmath,amssymb,amsfonts,enumerate,amsthm,graphicx,bm}
\usepackage[pagebackref]{hyperref}
\usepackage{color}
\usepackage{xcolor}
\usepackage[T1]{fontenc}
\usepackage[utf8]{inputenc}
\usepackage{setspace}
\usepackage{amscd}
\input xy
\xyoption{all}
\numberwithin{equation}{section} 
 \newtheorem{Theorem}{Theorem}[section] \newtheorem{Corollary}[Theorem]{Corollary} \newtheorem{Lemma}[Theorem]{Lemma} \newtheorem{Proposition}[Theorem]{Proposition} 
 \theoremstyle{definition}
  \newtheorem{Example}[Theorem]{Example} 
 \newtheorem{Remark}[Theorem]{Remark}

 \numberwithin{equation}{section}

 \DeclareMathOperator{\Spec}{Spec}  
   
\DeclareMathOperator{\codim}{codim}   \DeclareMathOperator{\Deg}{deg}
   
   \DeclareMathOperator{\Ht}{ht}
\DeclareMathOperator{\Reg}{reg}  \DeclareMathOperator{\grade}{grade}   \DeclareMathOperator{\A}{\alpha}

\def\xx{{\bf x}}
\def\yy{{\bf y}}

\def\TT{{\bf T}}

\def\ZZ{{\bf Z}}

\newcommand{\fm}{\mathfrak{m}}

\newcommand{\fp}{\mathfrak{p}}
\newcommand{\fq}{\mathfrak{q}}
\newcommand{\fa}{\mathfrak{a}}

\def\pp{{\mathbb P}}
\newcommand{\rar}{\rightarrow}
\newcommand{\lar}{\longrightarrow}
\newcommand{\llar}{-\kern-5pt-\kern-5pt\longrightarrow}
\newcommand{\surjects}{\twoheadrightarrow}

\def\restr{{\kern-1pt\restriction\kern-1pt}}

\def\geq{\geqslant}
\def\leq{\leqslant}

\def\codim{\hbox{\rm codim}}

\def\proj{\hbox{\rm Proj}}
\def\spec{\hbox{\rm Spec}}
\def\sym{\hbox{\rm Sym}}

\begin{document}

\title{Degree of Rational Maps versus Syzygies}
\author{M. Chardin\quad S. H. Hassanzadeh\quad A. Simis\footnote{Partially
		supported by a CNPq grant (302298/2014-2).}}

\date{}

\maketitle
%\tableofcontents
\footnotetext{Mathematics Subject Classification 2020
 (MSC2020). Primary   13A30, 13D02, 14E05
Secondary 13A02, 13H10, 13H15, 14A10.}

\begin{abstract}
One proves a far-reaching upper bound for the degree of a generically finite rational map between projective varieties over a base field of arbitrary characteristic. The bound is expressed as a product of certain degrees that appear naturally by considering the Rees algebra (blowup) of the base ideal defining the map.
Several special cases are obtained as consequences, some of which cover and extend previous results in the literature.

\end{abstract}

\section*{Introduction}

Let  $X\subseteq \mathbb{P}^n_k={\rm Proj}(k[x_0,\ldots,x_n])$ be a reduced and irreducible closed projective variety, where $k$ is a field of arbitrary characteristic. 
Let 
$$\Phi:X\dasharrow \mathbb{P}^m_k={\rm Proj}(k[y_0,\ldots,y_m])$$
be a generically finite rational map. 
The purpose of this work  is to give bounds for the degree $\deg(\Phi)$,
a numerical invariant that can be identified as the field degree of the extension of the corresponding function fields of $X$ and the (closure of the) image of $\Phi$.
It is common in algebraic geometry to deal with this invariant in terms of Chern classes.
In commutative algebra, it can be addressed in several other ways, including reformulations in terms of certain multiplicities.
Along this line, pretty much of the theory of rational/birational maps meets some important algebraic constructions and invariants.
 In this work the main discussion will be carried within this perspective; in particular, the obtained bounds are expressed in terms of these algebraic invariants.

The most basic fact at the outset is that $\Phi$ is represented by a set of forms $\{f_0,\ldots,f_m\}\subset R$ of a certain degree $d$ -- the so to say  {\em coordinates} of $\Phi$ -- where $R$ stands for the homogeneous coordinate ring of the given embedding $X\subseteq \mathbb{P}^n_k$.
Such a representation is unique up to proportionality of the coordinates with multipliers in the field of fractions of $R$.
The ideal $I:=(f_0,\ldots,f_m)\subset R$ is often called a {\em base ideal} of $\Phi$. It is this ideal, along with the various algebras associated to it, that will encompass most of the work.
This approach has been shown quite useful in the previous literature (see, e.g., \cite{BJ}, \cite{Si}, \cite{SiVi}, \cite{DHS}, \cite{KPU}).
 
 The main result of this paper is Theorem~\ref{TDegree}, along with its various consequences and special cases, some of which have been proved before. It gives upper bounds for $\deg(\Phi)$ and
 rests on a close study of the Rees algebra (blowup) $\mathcal{R}_R(I)$ of the base ideal $I$ and its main features. In particular, the essential statement draws upon the structure of the defining degrees of certain subideals of the defining ideal of $\mathcal{R}_R(I)$.
 
 Since the bound valid is for arbitrary maps, one cannot expect that it is attained without further assumptions on the base ideal. At the other end one obtains a lower bound in certain cases (Theorem~\ref{Tinverse}), in terms of other invariants such as the Castelnuovo--Mumford regularity and the so-called Jacobian dual rank of \cite{DHS}.
 
 The overall strategy of the main result is to reduce it to finding an upper bound for the multiplicity of a one-dimensional standard domain over the field of fractions of the homogeneous coordinate ring of the closed image of $\Phi$. Since the latter ring is identified with the special fiber (or fiber cone) of the base ideal, there ensues a subtle algebraic path to be followed in the argument.
 
 An important case of the theorem relates to the knowledge of the generating degrees of the syzygies of $I$. Upper bounds in terms of these degrees have been established before in particular cases of $I$, such as when it has dimension one or when it is Cohen--Macaulay of codimension two (see, e.g., \cite{SUV}, \cite{BCD}, \cite{CidSi}) -- exact formulas  of varying complexity are also known in terms of certain local multiplicities.
 A built-in reason that the upper bounds in Theorem~\ref{TDegree} typically fail to be attained is that they are a product structure, while exact formulas usually include differences of products.
 %Thus, it may be difficult to put them together even when $I$ is linearly presented, in which case both the present product-like bound and other approaches imply that $\Phi$ is birational, as was known before (see \cite{Si} and \cite{DHS}, and the respective lists of references).

The contents of the sections are as follows.

The first section brings in the required algebraic preliminaries and their mutual relationship. The basic ingredients are a standard graded algebra $R$ over an arbitrary field and an ideal $I\subset R$ generated by forms of a fixed degree. 

The second section is a collection of preliminaries regarding the degree of structures in certain situations and some prime avoidance like results regarding generators of homogeneous ideals of $R$.
The contents are pretty technical to be explained here, so the reader is referred to the details of the section.

The third section contains the main bulk of the paper.
The main result is Theorem~\ref{TDegree}, giving an upper bound for the degree of a generically finite rational map $\Phi:X\dasharrow Y$ as a product of the degree of $X$ and certain standard degrees coming from the Rees algebra $\mathcal{R}_R(I)$ of the base ideal $I\subset R$ of $\Phi$, where $R$ denotes the homogeneous coordinate ring of $X$.
Getting such degrees is pretty effective in the sense that they depend on the generating structure of the defining ideal $J$ of $\mathcal{R}_R(I)$ or of some subideal of $J$ over which $J$ is a minimal prime.
One of the consequences of the main theorem is an upper bound for $\deg(\Phi)$ in terms of the syzygy degrees of the base ideal $I$ or of certain graded submodules of maximal rank mapping to the syzygy module of $I$.
The section ends with some comparison with the details of the Jacobian dual matrix method as developed in \cite{DHS}. The ideas are then illustrated through a couple of examples.

The fourth section deals with lower bounds. These are expressed as sums involving the so-called Jacobian dual rank of $\Phi$ and the regularity of the one-dimensional  graded domain mentioned above.
It is shown, in particular, that the defect of the Jacobian dual rank of $\Phi$ relative to its maximal possible value is a lower bound for $\deg(\Phi)$. 

The last section recovers some of the consequences of Theorem~\ref{TDegree} in terms of the so-called {\em row ideals} introduced in \cite{EU}.

\section{Preliminaries}\label{prelims}

Let $R$ denote a Noetherian ring and $I\subset R$ an ideal.

The {\em Rees algebra}  of $I$ is defined as the $R$-subalgebra 
$$
\mathcal{R}_R(I):=R[Iu] = \bigoplus_{n\ge 0} I^nu^n  \subset R[u].
$$

Assume that $(R,\fm)$ is a standard graded algebra over a field $k$ with maximal irrelevant ideal $\fm=([R]_1)$. Then, just as for a local ring, one defines the {\em fiber cone} (or {\em special fiber}) of $I$ to be 
$$
\mathfrak{F}_R(I): = \mathcal{R}_R(I)/\fm\mathcal{R}_R(I),
$$
and the {\em analytic spread} of $I$, denoted by $\ell(I)$, to be the (Krull) dimension of $\mathfrak{F}_R(I)$.

Let $I \subset R$ be a homogeneous ideal generated by  forms $\{f_0,\ldots,f_m\} \subset R$ of the same degree $d>0$ -- in particular, $I=\left([I]_d\right)$.
Consider the bigraded $k $-algebra 
$$
S:=R \otimes_k  k[y_0,\ldots,y_m]=R[y_0,\ldots,y_m],
$$ 
where $k[y_0,\ldots,y_m]$ is a polynomial ring over $k$ and the bigrading is given by  ${\rm bideg}([R]_1)=(1,0)$ and ${\rm bideg}(y_i)=(0,1)$.
Setting ${\rm bideg}(u)=(-d, 1)$, then $\mathcal{R}_R(I)=R[Iu]$ inherits a bigraded structure over $k $.
One has a bihomogeneous $R$-homomorphism
$S \longrightarrow  \mathcal{R}_R(I) \subset R[u] ,
\quad y_i  \mapsto  f_iu.$
Thus, the bigraded structure of $\mathcal{R}_R(I)$ is given by 
$$
\mathcal{R}_R(I) = \bigoplus_{c, n \in \ZZ} {\left[\mathcal{R}_R(I)\right]}_{c, n} \quad \text{ and } \quad {\left[\mathcal{R}_R(I)\right]}_{c, n} = {\left[I^n\right]}_{c + nd}u^n.
$$
One is often interested in the $R$-grading of the Rees algebra, namely,
${\left[\mathcal{R}_R(I)\right]}_{c}=\bigoplus_{n=0}^\infty {\left[\mathcal{R}_R(I)\right]}_{c,n}$, and of particular interest is  
$$
{\left[\mathcal{R}_R(I)\right]}_{0}=\bigoplus_{n=0}^\infty {\left[I^n\right]}_{nd}u^n = k \left[[I]_du\right] \simeq k \left[[I]_d\right]=\bigoplus_{n=0}^\infty {\left[I^n\right]}_{nd}\subset R.
$$
Clearly, $\mathcal{R}_R(I)=[\mathcal{R}_R(I)]_0\oplus \left(\bigoplus_{c\geq 1} [\mathcal{R}_R(I)]_c\right)=[\mathcal{R}_R(I)]_0\oplus \fm \mathcal{R}_R(I)$. Therefore, one gets
\begin{equation}
\label{eq_isom_Rees_0_special_fiber}
k \left[[I]_d\right]\simeq {\left[\mathcal{R}_R(I)\right]}_0\simeq \mathcal{R}_R(I)/\fm \mathcal{R}_R(I)
\end{equation}
as graded $k $-algebras.
We will use this identification without further ado and write $[I]_d=I_d$ for lighter reading.

Finally, consider the canonical surjective map 
\begin{equation}\label{sym2Rees}
\sym_R(I)\surjects \mathcal{R}_R(I),
\end{equation} where the source denotes the symmetric algebra of $I$.
Fix algebra presentations  
 $\mathcal{R}_R(I)=S/J$ and  $\sym_R(I)=S/L$, where $S=R[y_0,\ldots,y_m]$ is as above and $L\subset J$ are bihomogeneous ideals of $S$ -- so named {\em presentation ideals}.
 It is well-known, or easy to see, that $L$ is generated by the biforms $g_0y_0+\cdots +g_my_m$ of bidegree $(*,1)$ such that $g_0f_0+\cdots +g_mf_m=0$ -- in other words, $L=I_1((y_0\cdots y_n)\cdot\psi)$ where $\psi$ is the defining matrix of a free $R$-presentation of $I$:
 \begin{equation}\label{presentingI}
 R^s\stackrel{\psi}{\lar} R^{m+1} \lar I\rar 0.
 \end{equation}
 It follows that $L$ is identified with the bihomogenous ideal $$(J_{(*,1)}):=\left(\bigoplus _{r\geq 1} J_{(r,1)}\right).$$ 
 From now on, assume that $R$ is a domain.
  
Moreover, if $I$ has a regular element $a$ then $J=L:(a)^{\infty}$.
Supposing, for simplicity, that $R$ is a domain then $J$ is an associated prime of $L$. But, since $\sym_{R_{\wp}}(I_{\wp})=\mathcal{R}_{R_{\wp}}(I_{\wp})$ for every $\wp\in \Spec R\setminus V(I)$, a power of $I$ kills the kernel of (\ref{sym2Rees}), hence any prime containing $L$ contains either $I$ or $J$.
It follows that $J$ is actually a minimal prime of $L$.

\begin{Remark}\rm \label{Rminprime}
	There are many graded $R$-submodules $N$ of Im$(\psi)$ in (\ref{presentingI}) such that $J$ is obtained by saturating the corresponding  bihomogeneous subideal of $L$ -- e.g., the submodule generated by the Koszul syzygies or, for that matter, any $N$ having maximal rank $m$.
For any such $N$ with corresponding subideal $\mathcal{I}\subset L$, one might assume as well that $J$ is a minimal prime thereof.
At the other end, by a similar token, any squeezed bihomogeneous subideal $L\subset \mathcal{I}\subset J$ will serve the same purpose.
\end{Remark}

This observation clarifies the nature of the hypothesis in the main theorem of this work. 

\section{Lemmata}

The following result will play an important role in the sequel.
Similar consideration has appeared in \cite[(ii), p. 251]{SUV}.

\begin{Lemma}\label{LHilDeg} Let $A=\bigoplus_{i=0}^{\infty}A_i$ denote a standard  graded Noetherian domain of dimension $1$ over a field $F=A_0$. Let $K(A)$ be the subfield of homogeneous fractions of degree zero in the field of fractions of $A$. Then $K(A)$ is a finite extension of $F$ of degree $$[K(A):F]=e(A)$$ 
where $e(A)$ is the Hilbert multiplicity of the ring $A$.
\end{Lemma}
 \begin{proof} Set $e:=e(A)$. Fix an integer $d\geq 1$ such that $\dim_F(A_i)=e$ for $i\geq d$ and let $\{a_1,\ldots,a_e\}$ be an $F$-vector basis for $A_d$. Letting $Q(A)$ denote the fraction field of $A$, one knows that $Q(A)=K(A)(H)$ for some transcendental linear form. Since $\dim(A)=1$, then ${\rm trdeg}_FQ(A)=1$ hence ${\rm trdeg}_FK(A)=0$. 
 	
 Clearly,  $\{a_1/H^d,\ldots,a_e/H^d\}\subset K(A)$ and they are linearly independent over $F$.
 Thus, it suffices to prove that $K(A)$ is spanned by $\{a_1/H^d,\ldots,a_e/H^d\}$ as an $F$-vector space. 
	
First, any fraction of the form $\beta/H^{ld}$, with $\beta\in A_{ld}$, $l\geq 1$, belongs to $F\langle a_1/H^d,\ldots,a_e/H^d\rangle$.  Indeed, by construction, 
$\dim_F(A_{ld})=e$ for any such, hence there exist $\gamma_1,\ldots,\gamma_e\in F$ such that $\beta=\sum \gamma_iH^{(l-1)d}a_i$. Clearly, then $\beta/H^{ld}=\sum \gamma_ia_i/H^d.$
	
The inverse fraction $H^{ld}/\beta$ belongs to $F\langle a_1/H^d,\ldots,a_e/H^d\rangle$ as well. To see this, note that $K(A)$ is algebraic over $F$, hence there exist $\theta_i\in F$ such that $$\frac{\beta}{H^{ld}}\left(\sum \theta_i(\frac{\beta}{H^{ld}})^i\right)=1.$$	
Since the terms inside the parenthesis belong to $F\langle a_1/H^d,\ldots,a_c/H^d\rangle$, so does $H^{ld}/\beta$ as well. 
	
The argument for an arbitrary homogeneous fraction $\beta/\gamma$ is now clear, by writing
	$\beta/\gamma=(H^{ld}/\gamma)(\beta/H^{ld})$, where $ld=\deg \beta=\deg \gamma$.
\end{proof}
%-----------------------------------------

The argument in the next lemma has the flavor of prime avoidance, but we could not find such a precise statement in the literature.

\begin{Lemma}\label{lcrucial} Let $ R=\bigoplus _{i \geq 0} R_{i}$  be a standard $\mathbb{N}$-graded  *local Noetherian ring
with  *maximal ideal $\fm$, where $R_0$ is a local  ring with maximal ideal $\fm_{0}$ and $R_0/\fm_0$ is infinite. Let  $\fa\subset R$ be an ideal  generated by $r$ forms of degrees $d_1 \geq \ldots \geq d_r\geq 1$. Given a    minimal prime $\fp\subset R$ of $R/\fa$ of height $m$, one has: 
\begin{enumerate}
	\item[{\rm(1)}]   There exists a subset $\{\A_1,\ldots,\A_r\}\subseteq \fa $  such that $\deg(\A_i)=d_i$ and  $\Ht(\A_1,\ldots,\A_i)_{\fp}=i$   for $1\leq i\leq m$.  In particular, $\fp$  is also minimal over $(\A_1,\ldots,\A_m)$.
	\item[{\rm(2)}]  Suppose in addition that $R$ is a domain and $m\geq 2$. Then, at the expense of having only inequalities $\deg (\alpha_i)\leq d_i$ throughout, one can find  a subset $\{\A_1,\ldots,\A_r\}\subseteq \fa$ such that 
	$\Ht(\A_1,\ldots,\A_{i-1},\A_r)_{\fp}=i$, for $1\leq i\leq m$, where $\alpha_i$ has been traded for $\alpha_r$.
	\item[{\rm(3)}]  Suppose that $R$  is a factorial domain, $R_0$ is a field, $m\geq 3$ and $d_r=1$.
	Then the result in {\rm (2)} can be improved to have $\Ht(\A_1,\ldots,\A_{i-2},\A_{r-1}, \A_r)_{\fp}=i$ for $2\leq i\leq m$.
\end{enumerate}  
\end{Lemma}
\begin{proof} As usual, for a graded $R$-module $M$ and an integer $j$,  $M_j$ stands for the $j$th graded component of $M$, a module over the ground local ring $R_0$.  Let $\{f_1,\ldots,f_r\}$  be a homogeneous generating set of $\fa$  with  $d_t=\deg f_t=d_t$  for all  $t$.
	
(1) We will induct on $m$. The case $m=0$ is vacuosly satisfied. Assume that $m \geq 1$ and let $P^1,\ldots,P^n\subset R$ denote the minimal primes of $R$ contained in $\fp$.
 
{\sc Claim.} $\fa_{d_1} \backslash \bigcup_{i=1}^{n} P^i_{d_1}\neq \emptyset$.
 
Since $\fp$ is minimal over $\fa$ and $m\geq 1$, one has $\fa \not \subseteq P^i$ and $P ^i \neq \fm$ for every $1\leq i \leq n$. 
By prime avoidance, let $i_j\in \{1,\ldots, r\}$ such that $f_{i_j} \notin P^i$.  In particular, we have  $P^i_{d_{i_j}}\not= R_{d_{i_j}}$ (recall that $d_{i_j}\geq 1$).
Set $c_i:=d_{1} - d_{i_j}$. Since $R$ is
standard $\mathbb{N}$-graded, for any graded prime ideal $P\subset R$ and any $\mu \geq 1$, $P_\mu\not= R_\mu$ if and only if $P_1\not= R_1$. Hence, $R_{c_i}\setminus P^i_{c_i} \neq \emptyset$ and for any $r_i \in R_{c_i}\backslash
	P^i_{c_i} $, $r_if_{i_j} \in \fa_{d_1}\backslash P^i_{d_1}$.
Thus,  $ \fa_{d_1}\neq P^i_{d_1}\bigcap {\fa}_{d_1}$ for
every $1\leq i \leq n$. By Nakayama's lemma,  $
	\fa_{d_1}\neq P^i_{d_1}\bigcap \fa_{d_1} + \fm_0 \fa_{d_1}$.   Since $R_0/\fm_0$ is an infinite field, then $ \fa_{d_1}\neq
\bigcup_{i=1}^n (P^i_{d_1}\bigcap \fa_{d_1} + \fm_0 \fa_{d_1}) $. In particular, $ \fa_{d_1}\backslash\bigcup_{i=1}^n P_{d_1}^i  \neq \emptyset$, as required in the claim.

Now, let $\A_1 \in \fa_{d_1}\backslash \bigcup_{i=1}^n P_{d_1}^i $. The avoidance-like result in the last step above can actually be accomplished by choosing $\A_1$  of the form $f_1+\sum_{i\geq 2}l_if_i$, for suitable  coefficients  $l_i\in R_{d_1-d_i}$. Therefore, the ideal $\fa/(\A_1)$ is generated by elements of degrees $d_2,\ldots,d_r$. Moreover, since $\A_1\not\in P^i$ for every $i$, $\fp/(\A_1)$ is a minimal prime of $\fa/(\A_1)$ of height $m-1$.  The result now  follows by induction.
	
(2) As $f_r\not= 0$ and $R$ is a domain, $f_r$ is a non zero divisor and one can apply the result of (1) to the ring $R/(f_r)$ and the image of $\fa$ in this ring, thus proving the assertion. Note that, since there is no assumption on the minimality of the generating set $\{f_1,\ldots,f_r\}$, it is possible that some of the $f_i$'s vanish modulo $(f_r)$. Thus in the argument of part (1), the chosen element $\A_1$ may have degree at most $d_1$, and so on. 
	
(3) Since $f_r$ is assumed to be a linear form in a factorial domain,  the ideal  $(f_r)$ is prime. Hence, applying (2) to the ring $R/(f_r)$ and the image of $\fa$ in this ring yields the  assertion.
\end{proof}

The previous lemma admits an expected formulation for the entire set of minimal primes of $R/\fa$:

\begin{Lemma}\label{lcrucialG} 
With the notation of {\rm Lemma~\ref{lcrucial}}, let   $\mathfrak{P}$ be the set of minimal primes of $R/\fa$.  Then
\begin{enumerate}
	\item[{\rm (1)}]   There exists a subset $\{\A_1,\ldots,\A_r\}\subseteq \fa $  with $\deg(\A_i)=d_i$ for all $1\leq i\leq r$, such that  $\Ht(\A_1,\ldots,\A_i)_{\fp}=i$ for every $\fp\in \mathfrak{P}$ and every $1\leq i\leq \Ht\fp$.
	\item[{\rm (2)}]  Suppose in addition that $R$ is a domain and $\Ht(\fa)\geq 2$. Then there exists  a subset $\{\A_1,\ldots,\A_r\}\subseteq \fa$, with $\deg (\A_i)\leq d_i$ for all  $1\leq i\leq r$,  such that 
	$\Ht(\A_1,\ldots,\A_{i-1},\A_r)_{\fp}=i$ for every $\fp\in \mathfrak{P}$ and every $2\leq i\leq \Ht\fp$.
\end{enumerate} 
\end{Lemma}
The proof follows the usual technique of extending to the finitely many minimal primes of $R/\fa$.
To wit, in the notation of the argument of Lemma~\ref{lcrucial}, let $N$ denote the number of generators of $\fa_{d_1}$ as an $R_0$-module. In the course of that argument, for a given $\fp\in \mathfrak{P}$,  there is an open subset $U_{\fp}$ of $\mathbb{A}_{R_0/\fm_0}^N$ such that any element $\A_1\in \fa_{d_1}$ generated by using the coordinates of any point in  $U_{\fp}$ satisfies the inductive procedure. In the present context, one can find an element $\A_1$ that simultaneously works for all $\fp\in P$ by considering the intersection $\bigcap_{\fp\in P} U_{\fp}$ of finitely many nonempty open sets.

An alternative is to adapt the proof of Lemma~\ref{lcrucial} by changing at the outset to the minimal primes of $R$ contained in set $\bigcup_{\fp\in \mathfrak{P}}\fp$ (of course, the containment condition is vacuous if $R$ is a domain).

%====================================================
Next, consider a closed subscheme $Z$  of $\mathbb P^n_k$.  Let ${\rm Supp}(Z)=Z_1\cup\ldots \cup Z_s$ be a decomposition of ${\rm Supp }(Z)$ into irreducible closed subschemes. In other words, the $Z_i$'s are the schemes defined by the minimal primes corresponding to the defining ideal of $Z$. We define $\delta(Z):=\sum_{i=1}^s\deg(Z_i)$ where $\deg(Z_i)$ is the standard definition of the degree of an irreducible variety (the multiplicity with respect to the maximal ideal $(\xx)$).  Notice that $\delta(Z)=\deg(Z)$ if and only if  $Z$ is equidimensional and reduced. 

\begin{Lemma}\label{Lrb} With the above notations. Let $Z$ be any closed subscheme  of $\mathbb P^n_k$ and let $H_1,\ldots, H_t$ be  hypersurfaces {\rm (}not necessarily reduced or irreducible{\rm )} in $\mathbb P^n_k$. Then 
	$$\delta(Z\cap H_1\cap\ldots\cap H_t)\leq \delta(Z)\prod_i\delta(H_i).$$
\end{Lemma}
\begin{proof} By induction it is enough to show the assertion for $t=1$.
Thus, set $H_1=H$ and let ${\rm Supp}(Z)=Z_1\cup\ldots \cup Z_s$ be the decomposition of ${\rm Supp}(Z)$ into irreducible closed subschemes. Then ${\rm Supp}(Z\cap H)=\bigcup_i {\rm Supp}(Z_i\cap H)$, hence the irreducible components of  ${\rm Supp}(Z\cap H)$ are among those of ${\rm Supp}(Z_i\cap H)$ for $i=1,\ldots, s $. Allowing for possible repetitions, it gives $\delta(Z\cap H)\leq\sum\delta(Z_i\cap H)$.
Similarly, let ${\rm Supp}(H)=H_1\cup\ldots \cup H_s$ be a decomposition of $H$ into irreducible closed subschemes. Then  $\delta(Z_i\cap H)\leq\sum\delta(Z_i\cap H_j)$.
Given indices $i,j$, either $\delta(Z_i\cap H_j)=\delta(Z_i)=\deg(Z_i)$ if $H_j\supset Z_i$, or else, the standard B\'ezout theorem  \cite[Theorem I.7.7]{Ha} gives
$$\delta(Z_i\cap H_j)=\sum_l \deg(T_l)\leq  \sum_l i(Z_i,H_j,T_l)\deg(T_l)=\deg(Z_i)\deg(H_j),$$
where the $T_l$'s are the irreducible component of ${\rm Supp}(Z_i\cap H_j)$.   
	
Therefore, $\delta(Z_i\cap H)\leq \sum_j \deg(Z_i)\deg(H_j)=\deg(Z_i)\delta(H)$, and hence $\delta(Z\cap H)\leq \sum_i \deg(Z_i)\delta(H)=\delta(Z)\delta(H),$ as was to be shown.
\end{proof}

\begin{Lemma}\label{Ldelta} Let $k$ be a  field and  $V\subset \pp^n_k$ be  a reduced irreducible closed  subscheme   and  $T\subseteq V$  be a closed subscheme defined by equations of degrees $d_1\geq\cdots\geq d_r\geq 1$.  %Set $p=\max_{i}\{{\rm codim}_X(T_i)\}$ where 
Let $T_i$'s be the irreducible components of ${\rm Supp}(T)$. Then, for any $p\leq r$, 
$$\sum_{\{i| {\rm codim}_V(T_i)\leq p\}}\deg(T_i)\leq d_1\cdots d_{p-1}d_r\deg(V).$$
%In particular, $\delta(T)\leq d_1\cdots d_{\dim(X)-1}d_r\deg(X)$
\end{Lemma}
\begin{proof} By the flat extension $k\subset k(u)$, where $u$ is a variable over $k$, one can assume that $k$ is infinite.

 Let $R$ be the homogeneous coordinate ring of  $V\subset \pp^n_k$ and let $\fa\subset R$ be the ideal of  definition of $T$ in $V$.  Let $\fp_i$ be the minimal prime ideal of $R/\fa$ defining the component $T_i$. Applying Lemma \ref{lcrucialG}, one can choose a subset $\{\A_1,\ldots,\A_{p-1},\A_r\}\subset \fa$ such that all $\fp_i$'s are minimal over $(\A_1,\ldots,\A_{\Ht(\fp_i)-1},\A_r)$. Since  $(\A_1,\ldots,\A_{\Ht(\fp_i)-1},\A_r)\subseteq (\A_1,\ldots,\A_{p-1},\A_r)\subseteq \fp_i$, $\fp_i$ must be minimal over $(\A_1,\ldots,\A_{p-1},\A_r)$. That is to say, every $T_i$ is an isolated component of $V\cap H_1\ldots\cap H_p$ where $H_i\subseteq  \pp^n_k $ is the reduced  hypersurface defined by $\A_i$. The result now follows from Lemma \ref{Lrb}.
\end{proof}

\begin{Remark}\label{Rfactorial}\rm
If furthermore $R$ is factorial and $p:=\codim T_i \geq 2$, then for any $i$,
$$\deg(T_i)\leq d_1\cdots d_{p-2}d_{r-1}d_r\deg(V).$$
Indeed, a minimal generator $f$ of the defining ideal of $T_i$, of minimal possible degree, satisfies $\deg(f)\leq d_r$ and generates a prime ideal. One can therefore replace $R$ by $R/(f)$ and apply the lemma.
\end{Remark} 

\section{Upper bounds}

\subsection{Basic notation}

Let $k$ denote a field of arbitrary characteristic and let $X\subseteq \mathbb{P}^n_k={\rm Proj}(k[x_0,\ldots,x_n])$ be a reduced and irreducible closed projective variety.
One is given a rational map 
$$\Phi:X\dasharrow \mathbb{P}^m_k={\rm Proj}(k[y_0,\ldots,y_m]),$$ 
with (closed) image $Y$. 

We will often  write $k[\xx]=k[x_0,\ldots,x_n]$ and $k[\mathbf{y}]=k[y_0,\ldots,y_m]$.
Given a closed projective (respectively, biprojective) variety $X\subset \pp_k^n$ (respectively, $\mathfrak{X}\subset \pp_k^n\times_k \pp_k^m$), we denote by $K(X)$ (respectively, $K(\mathfrak{X})$) its field of functions.
%  Notice that $K(X)$ and $K(Y)$ consist of homogeneous fractions of degree zero in the respective fields of fractions of the domains $R$ and the homogeneous coordinate ring of $Y\subset \mathbb{P}^m_k$. 

The {\em degree} of $\Phi$ $:X\dashrightarrow Y$ is the degree of the field extension $K(X)|K(Y)$.
 $\Phi$ is {\em generically finite} if $\left[K(X):K(Y)\right]<\infty$.

 The following result is well-known; we isolate it for the reader's convenience.
 
 \begin{Lemma}\label{Lshafar}
 	Let $\Phi:X\dashrightarrow \pp_k^m$ be a rational map with image $Y\subset \pp_k^m$. The following are equivalent:
 	\begin{enumerate}[\rm (i)] 
 		\item $\Phi$ is generically finite.
 		\item   $\dim(X)=\dim(Y)$.
 		\item  There exists a Zariski dense open subset $U \subset Y$ such that $\Phi^{-1}(U) \rightarrow U$ is a finite morphism.
 	\end{enumerate}
 \end{Lemma}

\smallskip

So much for the geometric side.

The map $\Phi$ is represented by a set of homogeneous elements $\{f_0,\ldots,f_m\}\subset R$ of a certain degree $d$ -- called {\em coordinates} of $\Phi$ -- where $R$ stands for the homogeneous coordinate ring of the given embedding $X\subseteq \mathbb{P}^n_k$.
Such a representation is unique up to proportionality of the coordinates within $R$.
The ideal $I:=(f_0,\ldots,f_m)\subset R$ is often called a {\em base ideal} of $\Phi$.

The degree $\deg(X)$ of the projective variety $X$ in its embedding is the graded Hilbert multiplicity $e(R)$ of the standard graded $k$-algebra $R$.

An essential role will be played in this work by the Rees algebra $\mathcal{R}_R(I)=R[Iu]$ of $I$ and by its special fiber $\mathcal{R}_R(I)/\fm\mathcal{R}_R(I)$, where $\fm=(R_1)$.
Recall from the preliminaries (Section~\ref{prelims}) that $\mathcal{R}_R(I)/\fm\mathcal{R}_R(I)\simeq k[I_d]$ as graded $k$-algebras.
Normalizing degrees, the corresponding projective subvariety of $\pp_k^m$ coincides with the closed image $Y$ of $\Phi$.

\begin{Remark}\rm
It is important to note heretofore that we will fix a base ideal of $\Phi$, but any two base ideals will have graded isomorphic respective Rees algebras and special algebras. 
This will give certain slight instability in the main result, but this is a small price to pay as far as applications are concerned.
\end{Remark} 

As a final piece of notation, recall from Section~\ref{prelims} that $\mathcal{R}_R(I)$ is a bigraded $k$-algebra with the induced bidegree of $k[\xx]\otimes_k k[\yy]=k[\xx,\yy]$. Thus, one can write  $\mathcal{R}_R(I)\simeq S/J$, where $S:=R\otimes_kk[\yy]=R[\yy]$, for some bihomogeneous presentation ideal $J$.
Again, $J$ is not uniquely defined by $I$, much less by $\Phi$, but its main numerical invariants are uniquely determined by $I$.

\subsection{Main theorem}

 \begin{Theorem}\label{TDegree} Let   $\Phi: X\dasharrow Y\subset \pp_k^m$ be a generically finite rational map with source a non-degenerate reduced and irreducible variety $X\subseteq \mathbb{P}^n_k$. Fix a base ideal $I\subset R$ of $\Phi$ and let $J\subset S$, as above, be a bihomogeneous presentation ideal of the Rees algebra $\mathcal{R}_R(I)$ of $I$.
 Let $\mathcal{I}\subset J$ be a bihomogeneous subideal, with a set of generators consisting of bihomogeneous elements of $\xx$-degrees $d_1\geq \cdots\geq d_r\geq 1$.
 If $J$ is a minimal prime of $\mathcal{I}$, then 
  $$\deg(\Phi)\leq d_1\cdots d_{t-1}\cdot d_r\cdot e(R),$$ 
    where $t=\dim X$.
\end{Theorem}
\begin{proof} First, note that the statement at least makes sense because $t\leq r$; indeed, since a finitely generated domain over a field is catenary, one has
\begin{eqnarray*}
&t = \dim X = \dim Y \leq m = \Ht (J) \\
&\leq r, \quad \mbox{by Krull's prime ideal theorem.}
\end{eqnarray*}
	
The idea of the proof consists in reducing the question to a one dimensional standard graded domain over a field and using Lemma~\ref{LHilDeg}.
We will first collect some overall facts about the present dealings, then introduce the role of the given subideal $\mathcal{I}$.

From the above prolegomena and the general preliminaries in Section~\ref{prelims}, one can write $\mathcal{R}_R(I)=[\mathcal{R}_R(I)]_0\oplus \fm \mathcal{R}_R(I)$, a direct decomposition of graded modules over $[\mathcal{R}_R(I)]_0\simeq k[I_d]$.
By a previous identification, the latter is the special fiber $\mathcal{F}(I)$ of $I$ and also the homogeneous coordinate ring of the closed image $Y$ of $\Phi$. In particular, the field of fractions of $[\mathcal{R}_R(I)]_0$ is identified with the field $Q(Y)$ of fractions of this ring.

For a lighter notation, set $B:=k[\yy]\subset S:=R[\yy]$ and take respective presentations  $[\mathcal{R}_R(I)]_0\simeq B/\fq$ and $\mathcal{R}_R(I)\simeq S/J$. In this notation, one has
\begin{eqnarray*}\label{redenoting}
	\mathcal{R}_R(I)\otimes_{[\mathcal{R}_R(I)]_0} Q(Y) &\simeq &
	\mathcal{R}_R(I)\otimes_{B/\fq}  B_q/qB_q\simeq \left(\mathcal{R}_R(I)\otimes_B B/\fq\right)\otimes_{B/\fq}  B_q/qB_q\\
	&\simeq & \mathcal{R}_R(I)\otimes_B \left(B/\fq\otimes_{B/\fq}  B_q/qB_q\right)\simeq \mathcal{R}_R(I)\otimes_B B_q/qB_q\\
	&\simeq &  \mathcal{R}_R(I)\otimes_B Q(Y)\simeq (S/J)\otimes_B Q(Y).
\end{eqnarray*}

 Moreover, $\mathcal{R}_R(I)\otimes_{[\mathcal{R}_R(I)]_0} Q(Y)$ is a one-dimensional domain; see for instance \cite[Proof of Theorem 5.3]{KPU}.

Let $\Gamma$ be (the closure of) the graph of $\Phi$, namely, $\Gamma={\rm BiProj}(\mathcal{R}_R(I))$.
Let $Q(\Gamma)$ by abuse denote the field of fractions of $\mathcal{R}_R(I)$, while $K(\Gamma)$ will denote the function field of $\Gamma$ as a biprojective variety, i.e., the subfield of $Q(\Gamma)$ consisting of the fractions of bihomogeneous elements of the same bidegree.

Note the usual diagram of maps:

\begin{equation}\label{diagram}
\xymatrix{
	& \Gamma \ar^{\pi}[rd] \ar [ld]_{\beta}                        &                   \\
	X \ar@^{-->} [rr]^{\Phi} &   & Y   \\    
}
\end{equation}

The map $\beta$ is a rational in the context of abstract non-embedded varieties, but one can retrieve its degree as $[K(\Gamma):K(X)]$. Likewise, the degree of $\pi$ is $[K(\Gamma):K(Y)]$ -- note that both degrees are well-defined integers since $\dim \Gamma=\dim \mathcal{R}_R(I)-2=\dim R-1=\dim X=\dim Y$. Since $\beta$ is birational, its degree is $1$, hence $\deg (\Phi)=\deg(\pi)$.

Thus, we can move over to the rational map $\pi$.

\smallskip

{\sc Claim 1.} $\deg(\pi)=[K((S/J)\otimes_B Q(Y)):Q(Y)]$ where $K((S/J)\otimes_B Q(Y))$ denotes the subfield of $Q(\Gamma)$ of fractions of homogeneous elements with same $\xx$-degree.

To see this, recall that $Q(Y)|K(Y)$ and $K(\mathcal{R}_R(I)\otimes_B Q(Y))|K(\Gamma)$ are both purely transcendental extensions of degree one, hence  $Q(Y)=K(Y)(H)$ and $K(\mathcal{R}_R(I)\otimes_B Q(Y))=K(\Gamma)(H)$ for a sufficiently general $k$-linear form $0\not= H\in B/\fq\subset Q(Y)$. 
Clearly, then 
$$\deg(\pi)=[K(\Gamma):K(Y)]=[K((S/J)\otimes_B Q(Y)):Q(Y)].$$

The above two claims imply, via Lemma~\ref{LHilDeg}, that $\deg(\pi)=e((S/J)\otimes_BQ(Y))$. 

Therefore, we have attained our preliminary goal of reducing the problem to estimating the multiplicity of a one-dimensional standard graded domain over a field.

%\textcolor{red}{\sc From this point down looks pretty thresh. Note that Claim 3 is not needed, as $\dim \bar{\Gamma}=0$ is already known from Claim 1.}
	
	It is now time to bring over the ideal $\mathcal{I}$. 
	%Since $\fq= J\cap B$,  one can move across $q$ when tensoring with $\otimes_BQ(Y)$, to bring it to the right factor, where it vanishes. Therefore, $\mathcal{R}_R(I)\otimes_B Q(Y)=(S/J)\otimes_B B_{\fq}$, where $S=R[\yy]$. By a similar token, $(S/\mathcal{I})\otimes_BQ(Y)=(S/(\mathcal{I},\fq))\otimes_BB_\fq$. 

%Further, since  $J$ is a minimal prime  of $\mathcal{I}$,  the ideal $J\otimes_B B_{\fq}\subset S\otimes_BB_{\fq}$  is a minimal  prime of  $(\mathcal{I},{\fq})\otimes_B B_{\fq}$.  Passing to the quotient ring
%$$S\otimes_BQ(Y)=\frac{S\otimes_BB_{\fq}}{\fq S\otimes_BB_{\fq}},$$ the image of $J\otimes_B B_{\fq}$ is still a minimal prime of the image of $(\mathcal{I},{\fq})\otimes_B B_{\fq}$.

%\smallskip

%{\sc Claim 3.} $\Ht_{S\otimes_B Q(Y)}((J\otimes_B B_{{\fq}})/(\fq S\otimes_BB_{\fq}))=t$.

%Recall that $t=\dim X=\dim R-1$.
%Since $R$ and $B=k[\yy]$ are in disjoint variables, one has $\dim R\otimes_k Q(Y)=\dim R$. 
%On the other hand 
%$$S\otimes_BQ(Y)=(R\otimes_k k[\yy])\otimes_{k[\yy]}Q(Y)=R\otimes _k Q(Y).$$
%We have seen that $\dim (S/J)\otimes_B Q(Y)=1$.
%Therefore 
%$$\Ht_{S\otimes_B Q(Y)}(J\otimes_B B_{{\fq}}/\fq S\otimes_BB_{\fq})=\dim S\otimes_BQ(Y)-\dim (S/J)\otimes_B Q(Y)=\dim R-1,$$ 
%as claimed.

One has natural inclusions of biprojective schemes over $\spec (k)$
$$\Gamma={\rm BiProj}_k(S/J)\subset \mathcal{Z}:={\rm BiProj}_k(S/\mathcal{I})\subset \mathcal{X}:={\rm BiProj}_k(S)\subset \mathbb{P}^n_k\times_k{\mathbb{P}^m_k},$$
inducing similar inclusions at the level of the associated projective schemes over $\spec(Q(Y))$:
\begin{eqnarray*}
\widetilde{\Gamma}={\rm Proj}_{Q(Y)}((S/J)\otimes_BQ(Y))&\subset & \widetilde{\mathcal{Z}}={\rm Proj}_{Q(Y)}((S/\mathcal{I})\otimes_BQ(Y))\\
&\subset & \widetilde{\mathcal{X}}={\rm Proj}_{Q(Y)}(S\otimes_BQ(Y)))\subset \mathbb{P}^n_{Q(Y)}.
\end{eqnarray*}

On the other hand, $S\otimes_BQ(Y)=(R\otimes_k k[\yy])\otimes_{k[\yy]}Q(Y)=R\otimes _k Q(Y)$, hence $\widetilde{\mathcal{X}}={\rm Proj}_{Q(Y)}(R\otimes_k Q(Y)))$.
Since the Hilbert function is stable under base field extension, one has $\deg(X)=\deg(\widetilde{\mathcal{X}})$.

Finally, note that $\widetilde{\Gamma}$ is a maximal irreducible component of $\widetilde {\mathcal{Z}}$. Indeed, since  $J$ is a minimal prime  of $\mathcal{I}$,  the ideal $J\otimes_B B_{\fq}\subset S\otimes_BB_{\fq}$  is a minimal  prime of  $(\mathcal{I},{\fq})\otimes_B B_{\fq}$. Then  
 the image of $J\otimes_B B_{\fq}$ is still a minimal prime of the image of $(\mathcal{I},{\fq})\otimes_B B_{\fq}$ over the residue ring $S\otimes_BQ(Y)=(S\otimes_BB_{\fq})/(\fq S\otimes_BB_{\fq}).$
 
 But, even more:

\smallskip
{\sc Claim 2.} $\Ht_{S\otimes_B Q(Y)}((J\otimes_B B_{{\fq}})/(\fq S\otimes_BB_{\fq}))\leq t$.

Recall that $t=\dim X=\dim R-1$.
Since $R$ and $B=k[\yy]$ are in disjoint variables, one has $\dim R\otimes_k Q(Y)=\dim R$. 
On the other hand 
$S\otimes_BQ(Y)=R\otimes _k Q(Y).$
Since $\dim (S/J)\otimes_B Q(Y)=1$, as seen above, then
$$\Ht_{S\otimes_B Q(Y)}(J\otimes_B B_{{\fq}}/\fq S\otimes_BB_{\fq})\leq \dim S\otimes_BQ(Y)-\dim (S/J)\otimes_B Q(Y)=\dim R-1,$$ 
as claimed.

Now apply Lemma~\ref{Ldelta} with $V=\widetilde{\mathcal{X}}$, $T=\widetilde{\mathcal{Z}}$ and $p=\Ht_{S\otimes_B Q(Y)}(J\otimes_B B_{{\fq}}/\fq S\otimes_BB_{\fq})$. 
Since $\widetilde{\mathcal{Z}}$ is defined in (standard) degrees $d_1\geq\cdots\geq d_r\geq 1$,  it follows that 
$$\sum_{\{i| {\rm codim}_{\widetilde{\mathcal{X}}}(\widetilde{\mathcal{Z}}_i)\leq p\}}\deg(\widetilde{\mathcal{Z}}_i)\leq d_1\cdots d_{p-1}d_r\deg(\widetilde{\mathcal{X}}),$$
where $\widetilde{\mathcal{Z}}_i$ is a maximal irreducible component of $\widetilde{\mathcal{Z}}$.
Now, as seen above, $\widetilde{\Gamma}$ is among these components and,  moreover, its codimension is at most $p$.
Therefore,
$$\deg(\Phi)=\deg(\pi)=\deg(\widetilde{\Gamma}))\leq d_1\cdots d_{p-1}\cdot d_r\cdot \deg(\widetilde{\mathcal{X}})\leq d_1\cdots d_{t-1}\cdot d_r\cdot \deg(X),$$
since $p\leq t$ and $\deg(X)=\deg(\widetilde{\mathcal{X}})$.
This completes the proof.
\end{proof}

\begin{Remark}\label{Rr-1}\rm
	(a) For the purpose of the stated bound in Theorem~\ref{TDegree}, the ideal $\mathcal{I}$ provides a  smaller bound  when it is generated by a subset of a minimal set of  bihomogeneous generators of $J$.
	
(b) If further $R$ is factorial and $t\geq 2$, it follows from Remark \ref{Rfactorial} that 
$$\deg(\Phi)\leq d_1\cdots d_{t-2}\cdot d_{r-1}d_r\cdot e(R).$$ 

(c) With the same notation as above, Kustin-Polini-Ulrich in \cite[Theorem 5.3]{KPU}  show  that $\deg(\Phi)=e(\mathcal{R}_R(I)_{(\xx\mathcal{R}_R(I))}).$  
Thus, Theorem~\ref{TDegree} provides an upper bound for the multiplicity of the local ring $\mathcal{R}_R(I)_{(\xx\mathcal{R}_R(I))}$.
\end{Remark}

%%-----------------------------------------------

\subsection{Consequences}\label{consequences}

A notable case of the previous theorem is when $\mathcal{I}$ is closely related to the presentation ideal of the symmetric algebra of $I$, in which case the bound is expressed in terms of syzygies.

\begin{Corollary}\label{Cmaxrank} {\rm (Syzygy rank criterion)} Let     $\Phi: X\dasharrow Y\subset \pp_k^m$ be a generically finite rational map with source a non-degenerate reduced and irreducible variety $X={\rm Proj}(R)\subseteq \mathbb{P}^n_k$. Fix a base ideal $I\subset R$ of $\Phi$ and its minimal graded presentation
	$$F_1=\bigoplus^{r}R(-d-d_i)\xrightarrow{\psi} F_0=\bigoplus^{m+1}R(-d)\to I\to 0.$$ 
	Let $G$ be a finitely generated graded $R$-submodule of ${\rm Im}(\psi)\subset F_0(+d)$ of rank $m$, generated in degrees $d_{i_1}\geq\cdots\geq d_{i_s}$. 	
	If $\grade_R(I)\geq 2$, then  $\deg(\Phi)\leq d_{i_1}\cdots d_{i_{t-1}}\cdot d_{i_s}\cdot e(R)$
	where $t=\dim X$. 
\end{Corollary}

%\textcolor{red}{\sc (The statement had to be changed to a more general hypothesis, otherwise cannot apply directly to 4.7 below)}

\begin{proof}
	Let $N$ be the $(m+1)\times  s$ matrix whose columns are the generators of $G$.
	First note that $r\geq s\geq m\geq \dim Y=\dim X=t$.
	Since $N$ has rank $m$, we can choose an  $(m+1)\times m$ submatrix $M$ of $N$ of rank $m$. 
	With a previous notation, let $J\subset S=R[\yy]$ denote a presentation ideal of the Rees algebra of $I$. Then one has inclusions of ideals 
	$$L_M:=I_1(\yy\cdot M)\subset L_N:=I_1(\yy\cdot N)\subset L:=I_1(\yy\cdot \psi)\subset J.$$

	{\sc Claim.} $J$ is a minimal prime over $L_N$.
	
It suffices to prove that $J$ is a minimal prime over $L_M$.	To see this, let $\Delta_0,\ldots,\Delta_m$ be the maximal minors of $M$. Since $\grade_R(I)\geq 2$,  \cite[Lemma 3.9]{HS} shows that  $hI=(\Delta_0,\ldots,\Delta_m)$, for some non-zero element $h\in R$. Let $\mathfrak{S}\subset R$ denote the multiplicatively closed set of powers of $h$.
Taking fractions, $\mathfrak{S}^{-1}I=\mathfrak{S}^{-1}(\Delta_0,\ldots,\Delta_m)$ and $\grade_{\mathfrak{S}^{-1}R}\mathfrak{S}^{-1}(\Delta_0,\ldots,\Delta_m)\geq 2$ as well, yielding the free presentation (resolution)
	$$0\rightarrow\bigoplus^m \mathfrak{S}^{-1}R\xrightarrow{M} \bigoplus^{m+1}\mathfrak{S}^{-1}R\to \mathfrak{S}^{-1}I\to 0.$$
	 Thus, $\mathfrak{S}^{-1}(L_M)$ is the defining ideal of the symmetric algebra of $\mathfrak{S}^{-1}I$. But, $\mathcal{R}_{R}(hI)=\mathcal{R}_{R}(I)$ as $R$ is a domain. Therefore, $\mathfrak{S}^{-1}J$ is a presentation ideal of the Rees algebra of $\mathfrak{S}^{-1}I$ on $\mathfrak{S}^{-1}R$, and hence $\mathfrak{S}^{-1}J$ is a minimal prime over $\mathfrak{S}^{-1}q(L_M)$. Since  $h\not \in J$, $J$ is a minimal prime over $L_M$  thence over $L_N$, too. 
	
	Now the result follows from   Theorem~\ref{TDegree}.
\end{proof}

%============================================Noether Normalization
Curiously, Corollary \ref{Cmaxrank} provides a numerical obstruction for integrality in the context of a Noether normalization.

\begin{Corollary} Let $R$ be a standard graded Noetherian ring  of dimension $t+1$ over a field $k$  and let $\{f_0,\ldots f_t\}\subset R$ be a set of algebraically independent elements of the same degree $d$ generating an ideal of grade at least $2$. Let $N$ be a submatrix of maximal rank of the syzygy matrix of $(f_0,\ldots f_t)$. 
	If $k[f_0,\ldots f_t]\subset R$ is an integral extension then $\sqrt[t]{d_1\cdots d_t}\geq d$ where  $d_1,\ldots, d_t$ are the column degrees of $N$.
\end{Corollary}
\begin{proof} Set $X$ =Proj$(R)$ and let $\Phi:X\dasharrow \mathbb{P}^t$ denote the rational map defined  by ${f_0,\ldots, f_t}$.  If $A:=k[f_0,\ldots f_t]\subset R$ is an integral extension then, in particular, $\Phi$ is a generically finite map.
	
	% c.f. \cite[Proposition 10.9]{D}. 
Now, Corollary~\ref{Cmaxrank} implies the inequality $d_1\cdots d_te(R)\geq \deg(\Phi).$ 
On the other hand, $\deg(\Phi)=e(R^{(d)})= e(R)d^t$ by an application of \cite[Observation 2.8 ]{KPU} to the inclusion $A\subset R^{(d)}$, where $R^{(d)}$ is the $d$-th Veronese subring of $R$. Therefore, $\sqrt[t]{d_1\cdots d_t}\geq d$, as desired.
\end{proof}

%%================================
The next result profits from the existence of common factors among the generators of a base ideal. We give but the simplest case.
 
\begin{Corollary}  In the setting of 
	{\rm Corollary~ \ref{Cmaxrank}}, assume in addition that $f_0$ and $f_1$ have a common factor of degree $\delta$. Then $\deg(\Phi)\leq (d-\delta)d^{m-1}\cdot e(R).$
\end{Corollary}
\begin{proof} In the setting of Corollary~\ref{Cmaxrank}, take $G$ to be the submodule generated by the columns of the reduced Koszul relations
$$N:=
\begin{pmatrix}
f_1/h&f_2&\ldots&f_m\\
-f_0/h&0&\ldots&0\\
0&-f_0&\ldots&0\\
\vdots&&&\vdots\\
0&0&\ldots&-f_0
\end{pmatrix},
$$
where $h$ is the given common factor.
\end{proof}

 It is  known that a regular sequence of forms of degree $\geq 2$ cannot be the coordinates of a birational map. 
The degree of the rational map it defines comes out as an equality in the previous corollary. This value  is known (see, e.g., \cite[Observation 3.2]{KPU}). We give another argument in the spirit of this work. 
%$\%${\it Hamid: But I could not find any place where the degree of a complete intersection is computed!! Notice that, below, the field is not algebraically closed and $X$ is not the whole projective space. Otherwise one says that this is the classical Bezout theorem.}

\begin{Proposition}\label{PCI} Let $R$ denote a standard graded domain over a field $k$, with $\dim R=m+1\, (m\geq 1)$ and ${\rm edim}(R)=n+1$. Set $X:={\rm Proj}(R)\subset \mathbb{P}^n_k$. Let $\{f_0,\ldots,f_m\}\subset R$ be a regular sequence of forms of degree $d\geq 1$ defining a rational map  $\Phi: X\dasharrow \pp_k^m$ . Then $\deg(\Phi)=e(R)d^{m}$.
\end{Proposition}
\begin{proof} Since $\{f_0,\ldots,f_m\}$ is a regular sequence, these forms are algebraically independent over $k$, hence the map is generically finite, regular and dominant.
	
%	Let $R$ be the coordinate ring of $X$ and let  $I=(f_0,\ldots,f_m)\in R$ be a base ideal generated in degree $d$.  Since $\Phi$ is regular, $\codim I=\dim(R)$, and since it is dominant and  finite, $\dim(R)=m+1$. Thus,  $\{f_0,\ldots,f_m\}$  is a complete intersection, since $R$ is Cohen-Macaulay. In particular, $\grade I=m+1\geq 2$, hence the base ideal is uniquely defined (\cite[Proposition 1.1]{Si}), and so is the integer $d$.
	%According to the proof of Theorem \ref{TDegree}, one has to compute the multiplicity of  $\mathcal{R}_R(I)\otimes_B Q(Y)$. 
One has
$$\mathcal{R}_R(I)=\mathcal{S}_R(I)=\frac{R[y_0,\ldots,y_m]}{(f_iy_j-f_jy_i)}.$$
Under the notation in the proof of Theorem~\ref{TDegree},  tensor with $Q(Y)$ over $B=k[f_0,\ldots,f_m]$, so as to allow  inverting $y_i$'s. Thus, e. g., modulo the ideal $(f_jy_0-f_0y_j|j=1,\ldots,m)$, every generator $f_iy_j-f_jy_i$ with $i,j\neq 0$ is zero. This shows that $\mathcal{R}_R(I)\otimes_B Q(Y)$ is a quotient of $R[y_0,\ldots,y_m]\otimes_B Q(Y)$ by a regular sequence  of length $m$, generated in $\xx$-degree $d$. Thus $e(\mathcal{R}_R(I)\otimes_B Q(Y))=e(R)d^m$ by a well-known formula and  the result follows from the same argument as in the proof of Theorem~\ref{TDegree}.
\end{proof}

\begin{Remark}\label{dimension_one_formulas} The above is the case of a zero-dimensional base ideal. Equalities with base ideal of dimension $1$ have been proved by various authors. For example, the following formula is given in \cite[Theorem 3.3]{BCD}  for  a generically finite rational map
 $\Phi: X\dasharrow Y$ with base ideal of dimension $1$, generated in degree $d$:
  $$d^{t}e(R)- \sum_{\fp \in \proj(R/I)} [\kappa(\fp):k]e_{I_{\fp}}(R_{\fp})=\deg(\Phi)e(S)$$ where $R$ and $S$ are the coordinate rings of $X$ and $Y$ and $t=\dim(X)$. A similar result can be found in \cite[Theorem 2.5]{BJ} and \cite[Theorem 6.6]{SUV}. In the case of polar maps this formula has been obtained in \cite{DP} and \cite{FM}.
  In addition, the conjecture of Dimca for base ideal of dimension one has been proved by J. Huh (\cite[Theorem 4]{Huh}), namely, to the effect that no hypersurface of degree $\geq 3$ with only isolated singularities is homaloidal.
  In $\dim\geq 2$, there is a beautiful general formula in \cite[Theorem 4.1]{Xie} generalizing the formulas of \cite{SUV} by introducing the $j$-multiplicity.
  \end{Remark}
%%=======================================

%\textcolor{red}{\sc Up to here}

%---------------------
\subsection{Relation to the Jacobian dual rank}\label{Jacobian_dual}

In order to state a few more consequences, we now relate the present approach to the {\it Jacobian dual} method as introduced in detail in \cite{DHS}.

Recall that $B=k[\yy]\subset S=R[\yy]$, $J\subset S$ was a presentation ideal of the Rees algebra $\mathcal{R}_R(I)$ and $\fq\subset B$ was a presentation ideal of the special fiber of $I$. By the identification discussed in Section~\ref{prelims}, one can view $\fq\subset J$ sitting in bidegrees $(0,r)$, for $r\geq 1$, and the quotient field of the homogeneous coordinate ring of the closed image of $\Phi$ (after degree normalization) is identified with the quotient field $B_{\fq}/\fq B_{\fq}$ of the special fiber.

Consider the ideal $(J/\mathfrak{q} S) \otimes_B B_{\mathfrak{q}}\subset (S/\mathfrak{q} S) \otimes_B B_{\mathfrak{q}}\simeq S\otimes_B Q(Y)$, which is naturally graded over $Q(Y)$.
The $Q(Y)$-vector space $(J/\mathfrak{q} S) \otimes_B B_{\mathfrak{q}})_1$ in degree one will play a role in the next result.
Essentially, its elements  are the $\xx$-linear bihomogeneous forms among the defining equations of the Rees algebra of $I$.
Its dimension $s$ is an invariant closely related to the Jacobian dual rank.

\begin{Remark}\label{Rjdrank} The invariant $s:=\dim_{Q(Y)}((J/\mathfrak{q} S) \otimes_B B_{\mathfrak{q}})_1$ is at least the linear rank of the presentation matrix of $I$ (see Corollary~\ref{Cmaxrank}).
As a consequence, Corollary \ref{CLlinearrank} below  will provide another proof of \cite[Theorem 3.2]{DHS}.
If $X\subseteq \mathbb{P}^n$ is non-degenerated, then in fact $s$ coincides with the Jacobian dual rank. It is shown in  \cite[Thorem 2.15]{DHS}  that $s\leq n+\dim(Y)-\dim(X)=n$ with equality if and only if $\Phi$ is birational.  Corollary \ref{CLlinearrank} also provides an alternative argument for the implication $(b) \Rightarrow (a)$ in \cite[Theorem 2.18]{DHS} in the case where $X=\mathbb{P}^n$.
\end{Remark}

\begin{Corollary}\label{CLlinearrank} With the above notation and the hypotheses of {\rm Theorem~\ref{TDegree}} one has 
$$
\deg(\Phi)\leq
\begin{cases}
d_1\cdots d_{t-s}\cdot e(R), {\text ~~if~~} s\leq t-1;\\
e(R), {\text ~~if~~} s\geq t,
\end{cases}
$$
where $t=\dim X$.

\end{Corollary}
\begin{proof}
%Notice that the integer $t$ appeared in Theorem \ref{TDegree}  is indeed ${\rm codim}((S/J)\otimes_B B_{\fq})$ as it is proved in Claim 3 in the course of the proof. So one can take out the $s$ linear forms in $S\otimes_BQ(Y)$ and substitute $t$ by $t-s$. Another proof would be as follows: 
This is actually, a consequence of the proof of the main theorem.
Let $L\subset J\otimes_BB_{\fq}$ be the subideal generated by the residues of the $\xx$-linear bihomogeneous forms from $J$.  Since $J$ is a minimal prime of $\mathcal{I}$ by assumption,  $J\otimes_BB_{\fq}$ is a minimal prime of $\mathcal{I}\otimes_BB_{\fq}$, hence of $(L,\mathcal{I}\otimes_BB_{\fq})$ as well. Now apply,  Lemma~\ref{lcrucial}, item (3), repeatedly $s$ times in the proof of Theorem~\ref{TDegree}.
\end{proof}

\begin{Remark}\label{Rbetterbound} One may notice that Lemma~\ref{lcrucial}, item (3), and Theorem~\ref{TDegree} give an improved upper bound, 
$$\deg(\Phi)\leq d_1\cdots d_{t-s-1}d_r\cdot e(R), {\text ~~if~~} s\leq t-1,$$
provided  $\mathcal{I}_{d_r}\not\subseteq L$, where $L$ is as above. 
\end{Remark}

%%==========================================
We next give some simple examples regarding the upper bound established in Corollary~\ref{CLlinearrank}.

\begin{Example} 
 (Complete intersection base ideal) Here the upper bound is attained, as comes out of the proof of Proposition \ref{PCI}. 
\end{Example}
\begin{Example}\label{mon_example} (Monomial base ideal)

(a) Consider the map $\Phi:\mathbb{P}^2_k\dasharrow \mathbb{P}^2_k$ defined by the monomials $\{x^2,yz,z^2\}$. 
For maps defined by monomials, \cite[Proposition 2.1]{SiVi} gives an easy criterion of birationality in terms of the $\gcd$ of the maximal minors of the corresponding log-matrix, namely, that the latter has to be $\pm d$, where $d$ is the common degree of the monomials. In this case it is nearly trivial to see that $\Phi$ is not birational because the matrix is square of determinant $4$.
Alternatively, a computation with Macaulay2 will give that the base ideal $I=(x^2,yz,z^2)$ has only one linear syzygy and is of linear type. Therefore, by \cite[Proposition 3.2]{DHS}, $\Phi$ is not birational. It also follows that $s=1$, taking for granted that $s$ coincides with the Jacobian dual rank of $I$.

As for the degree of $\Phi$, one can easily sees that a general point such as $(1:ab:b^2)$ has exactly two pre-images, so that $\deg(\Phi)=2$.
Alternatively, one can directly compute the field extension degree $[k((x,y,z)_2):k(x^2,yz,z^2)]$. Namely, one has $k((x,y,z)_2)=k(x^2,y/x,z/x)$ and $(z/x)^2=z^2/x^2\in k(x^2,yz,z^2)$, while $y/x=(yz/z^2)(z/x)$. Therefore, $[k((x,y,z)_2):k(x^2,yz,z^2)]=2$.
Finally, the full minimal graded free resolution of  $(x^2,yz,z^2)$ is 
$$0\rar R(-5)\rar R(-4)^2\oplus R(-3)\rar R(-2)^3\rar 0.$$
 Therefore, according to Corollary \ref{CLlinearrank}, $\deg(\Phi)$ attains the upper bound  $d_1e(R)=2\cdot 1=2$. 
We note that the improved bound in Remark~\ref{Rbetterbound} is not attained here. Actually, the obstruction $\mathcal{I}_{d_r}\subseteq (L)$ takes place, where $L$ is the presentation ideal of the symmetric algebra
%\textcolor{red}{\sc (But one has to say who is $\mathcal{I}$ in this example!)}

(b)
Let $\Phi:\mathbb{P}^2_k\dasharrow \mathbb{P}^2_k$ be defined by $\{x^2y,xz^2,y^2z\}$ This belongs to the class of monomial rational maps whose base ideal has radical $(xy,xz,yz)$. It can be shown that the only birational one (i.e., Cremona) is the classical involution defined by $\{xy,xz,yz\}$. It is an easy exercise to see that the syzygies of $\{x^2y,xz^2,y^2z\}$ are generated by the three reduced Koszul relations, all of standard degree $2$. On the other hand, $I=(x^2y,xz^2,y^2z)$ is an almost complete intersection that is a complete intersection locally at its minimal primes. 
Therefore, it is an ideal of linear type (\cite[Proposition 5.5 and proposition 9.1]{HSV}). It follows that $s=0$, hence the bounds give $\deg(\Phi)\leq 2.2.1=4$.
Finally, by Corollary~\ref{J_rank_null}, $\deg(\Phi)\geq 3$.
To decide between $3$ or $4$, requires a direct calculation. Namely,                                                                                                                                                                                                                                                                                                                                                                                                                                                                                                                                                                                                                                                       $y/x$ satisfies the equation $T^3-(y^2z)^2/(xz^2\cdot x^2 y)$
 of degree $3$ over $k(x^2y,xz^2,y^2z)$. This equation is irreducible as $y/z\notin k(x^2y,xz^2,y^2z)$ -- to see this, otherwise, $x^2z=(x^2/y^2)y^2z\in  k(x^2y,xz^2,y^2z)$, hence also $z/x=xz^2/x^2z\in k(x^2y,xz^2,y^2z)$.
 Since, by a similar token, $x^3=x^2y(x/y)\in k(x^2y,xz^2,y^2z)$, we would have $k((x,y,z)_3)=k(x^3,y/x,z/x)=k(x^2y,xz^2,y^2z)$, which would say that $\Phi$ is birational.
 To conclude, since $4$ is not a multiple of $3$, one must have $\deg(\Phi)=3$. Thus, the upper bound of Corollary~\ref{CLlinearrank} is missed by $1$.
\end{Example}

%===============================================

It may be interesting to stress the case where the source is two-dimensional.
\begin{Corollary}\label{Cdim2} Keep the same assumptions as in {\rm Remark~\ref{Rjdrank}}, let moreover $\dim X=2$. 
Let $d_1$ and $d_r$ respectively denote the highest and the lowest degrees among the degrees of a minimal set of generating syzygies of a base ideal.  Then 
$$\deg(\Phi)=
\begin{cases}
=1, {\text ~~if~~} s\geq 2;\\
\leq d_1e(R), {\text ~~if~~} s=1;\\
\leq d_1d_re(R) {\text ~~if~~} s=0.
\end{cases}
$$
 \end{Corollary}

Corollary \ref{Cdim2} generalizes the inequality assertion of the following result: 
\begin{Proposition} {\rm (\cite[Proposition 5.2]{BCD})} Let $\Phi: \mathbb{P}^2\dasharrow \mathbb{P}^2$ be a dominant rational map with a dimension $1$ base ideal $I$ that is moreover saturated, with minimal graded free resolution
$0\rar  R(-d-\mu_1)\oplus R(-d-\mu_2)\rar R(-d)^3\rar 0.$
Then, $\deg(F)\leq \mu_1\mu_2$ with equality if and only if $I$ is locally a complete intersection at its minimal primes.
\end{Proposition}

See \cite[Theorem 6.3]{CidSi} for more encompassing work assuming the base ideal is perfect of codimension two satisfying a certain condition on the local number of generators. 
%%%%%%%%%%%%%%%%%%%%%%%%%%%%%%%%%%%%%%%%%%%%%%%%%%%%%%%%%%%%%%%%%%%
\section{Lower Bounds}

There are known lower bounds for the degree of a polar rational map for isolated singularities (see, e.g., \cite[Theorem 1]{Huh}). In the case of arbitrary rational maps, any exact formula such as the ones quoted in Remark~\ref{dimension_one_formulas} implies trivially both upper and lower bounds. A different story is to recover these trivial bounds in terms of other interesting and more concrete invariants. In this part we wish to communicate a lower bound in the spirit of this work, bringing up some computable invariants. For this effect we will draw upon a known affirmative case of the Eisenbud-Goto conjecture (\cite{EG}).

 Recall the notation employed in {\rm Theorem~\ref{TDegree}} and in its proof. Take a presentation $$(S/J)\otimes_B Q(Y)\simeq Q(Y)[x_0,\ldots,x_n]/\mathfrak{J}.$$ 
In the notation of Subsection~\ref{Jacobian_dual}, let  $s=\dim_{Q(Y)}((J/\mathfrak{q} S)\otimes_B B_{{\fq}})_1$ denote the dimension of the  $\xx$-linear forms. As  stated in Remark~\ref{Rjdrank}, $s$ coincides with the Jacobian dual rank ${\rm jdrank}(\Phi)$ of $\Phi$. 
We have seen that $(S/J)\otimes_B Q(Y)$ is a one dimensional domain, hence is in particular Cohen--Macaulay.

%%====================Theorem 

\begin{Theorem}\label{Tinverse} In the notations and hypotheses of {\rm Theorem~\ref{TDegree}}, 
$$\deg(\Phi) \geq n+1-{\rm jdrank}(\Phi)+\Reg(\mathfrak{J} )-2$$
where $\Reg(\mathfrak{J})$  is the Castelnuovo--Mumford regularity of $\mathfrak{J}$. 

In particular, one always has $\deg(\Phi)\geq n+1-{\rm jdrank}(\Phi)$. 
\end{Theorem}
\begin{proof} According to the proof of Theorem \ref{TDegree}, to compute $\deg(\Phi)$ one has to compute the multiplicity of  the one dimensional domain  $(S/J)\otimes_B Q(Y)$.
% The equality comes after \cite[Ramek in page 111]{EG}.  

We notice that  the  Eisenbud-Goto conjecture holds for Cohen-Macaulay rings \cite[Corollary 4.15]{Eis}. Therefore,
 $$e((S/J)\otimes_B Q(Y))\geq \Reg((S/J)\otimes_B Q(Y))+{\rm edim}((S/J)\otimes_B Q(Y))-\dim((S/J)\otimes_B Q(Y)).$$
 Since $\Reg((S/J)\otimes_B B_{\fq})=\Reg(\mathfrak{J})-1$ and $s={\rm jdrank}(\Phi)$, it suffices to show that ${\rm edim}((S/J)\otimes_B Q(Y))=n+1-s$.
  To this end, note that
$$(S/J)\otimes_B Q(Y)\simeq \frac{S\otimes_BQ(Y)}{((J/\mathfrak{q} S)\otimes_B B_{{\fq}})}, $$ 
while $S\otimes_BQ(Y)\simeq R\otimes_kQ(Y)$ is non-degenerated since $R$ is so. Thus, canceling  $s$ linear forms out of the denominator,  the sought embedding dimension  is  $n+1-s$. 

As for the last claim, notice that if $\mathfrak{J}$ consists merely of linear forms  then ${\rm jdrank}(\Phi) =n$, hence in this case $\Phi$ is a birational map by \cite{DHS}, so the statement is trivial. Otherwise, $\Reg(\mathfrak{J})\geq 2$ and the assertion follows from the proved inequality. 
\end{proof}

\begin{Remark}
The lower bound of Theorem~\ref{TDegree} is attained in Example~\ref{mon_example} (b). Indeed, ${\rm jdrank}(\Phi)=1$ since the base ideal is of linear type with minimal syzygy degree $\geq 2$, while ${\rm reg}(\mathfrak{J})=2$ as $\mathfrak{J}$ is linearly presented perfect of codimension $2$.
\end{Remark}

\begin{Corollary}\label{J_rank_null} In the notation and hypotheses of {\rm Theorem~\ref{TDegree}}, if $J_{(1,\ast)}= 0$ then    $\deg(\Phi)\geq n+1 $. 
\end{Corollary}
\begin{proof} If   $J_{(1,\ast)}= 0$ then ${\rm jdrank}(\Phi)=0$,  in the notation of the Theorem \ref{Tinverse}. Therefore $\deg(\Phi)\geq n+1$.
\end{proof}

%\begin{Remark}\rm
	Let us note that an immediate application of the above corollary is to the case where the base ideal $I$ is generated by a regular sequence of $d$-forms, with $d\geq 2$.
	In general, the result of the corollary was only known in the case where $\Phi$ is birational onto its image (see \cite{DHS}).
%\end{Remark}

%==========+Corollary
%The one-dimensional ring $\mathcal{R}_R(I)\otimes_BQ(Y)$ has performed an essential role so far, both for upper and lower bouns of $\deg(\Phi$).  
 The next result provides an upper bound for the regularity of  $\mathfrak{J}$.

\begin{Proposition}\label{CP} In the above notation and the hypotheses of {\rm Theorem~\ref{TDegree}}, assume that $X=\pp ^n$. 
	Then
$$\Reg(\mathfrak{J})\leq d_1+\ldots+d_{n-2}+d_{r-1}+d_r-n+1$$
\end{Proposition}
\begin{proof} According to Lemma  \ref{lcrucial} and Remark \ref{Rfactorial}, the ideal $\mathfrak{J}$ is defining a subscheme $\tilde{\Gamma}$ of the zero dimensional part $\widetilde{Z}$, defined by equations of degrees $d_1,...,d_{n-2},d_{r-1},d_r$ as in Remark \ref{Rr-1}(b). 
Therefore, $\widetilde{Z}$ is locally a complete intersection at each point of $\tilde{\Gamma}$.

Therefore, according to  \cite[Theorem A]{CP}, the claimed inequality holds.
\end{proof}

%==================================================++SUV
%%%%-----------------------------------------------------------------------

\section{Alternative approach}
%%%%==================================================================================

Although syzygies are very important, they are not the only data covered by Theorem~\ref{TDegree}. Actually, if one is only interested in an upper bound given by the product of the highest degrees of the syzygies of $I$, then there will be other ways to see this fact, although obtaining the same upper bound has more cost.

Here, we recall the notion of a {\em row ideal} introduced in \cite{EU}.   Let $R$ be a  Noetherian ring and  $I=(f_0,\ldots,f_m) $ be an ideal of $R$ with a presentation matrix $\psi$. Let 
 $\xi$ be the row matrix  $[q_0\ldots q_m]$ with entries in $R$. The (generalized) row ideal of $I$ corresponding to $\xi$ is the ideal $I_1(\xi\cdot \psi)$.

Consider the polynomial ring $R[\TT]=R[T_0,\ldots, T_m]$ and a specialization map $R[\TT]\rar R$ with $T_i\mapsto q_i$.
Given elements $\{f_0,\ldots,f_m\}\subset R$, let $I_q\subset R$ denote the ideal of $R$ generated by the image of the Koszul relations $\{f_iT_j-f_jT_i| 0\leq i<j\leq m\}$.

 The following Lemma is a basic fact about row ideals. We provide a proof for the sake of completeness.

\begin{Lemma}\label{Lrow} Let $R$ be a  Noetherian ring and let $I=(f_0,\ldots,f_m) $ be an ideal of $R$ with a presentation matrix $\psi$. In the above notation, if at least one of the $q_i$'s is a unit, then $$I_1(\xi\cdot \psi)=I_q:_RI,$$
	where $\xi$ is the row matrix  $[q_0\ldots q_m]$.
\end{Lemma}
\begin{proof}
 Let $\psi=(a_{ij})$ be an $(m+1)\times r$ matrix presenting $I$. Let $\mathcal{J}=(j_1,\ldots,j_r):=I_1(\xi\cdot \psi)$. Note that $j_t=\sum_{i=0}^rq_ia_{it}$.

 For any $f_k\in I$, one has
\begin{equation*}
j_tf_k =\sum_{i=0}^mq_ia_{it}f_k\equiv \sum_{i=0}^mq_ka_{it}f_i \bmod I_q
\end{equation*} 
But $\sum_{i=0}^mq_ka_{it}f_i=q_k \sum_{i=0}^ma_{it}f_i$ vanishes since $(a_{0t},\ldots,a_{mt})$ is a syzygy of $f_0,\ldots,f_m$.
This shows the inclusion $I_1(\xi\cdot \psi)\subseteq I_q:_RI$.

To see the reverse inclusion, assume without loss of generality that $q_0$ is unit. Let $s\in (I_q:_RI)$. Writing out  $sf_0\in I_q$ gives 
$$sf_0=b_1(q_1f_0-q_0f_1)+\ldots+b_m(q_mf_0-q_0f_m),$$
for certain $\{b_1,\ldots,b_m\}\subset R$.
Hence, $(s-\sum_{i=1}^m b_iq_i)f_0+\sum_{i=1}^mq_0b_if_i=0$, i.e.,
the vector $(s-\sum_{i=1}^mb_iq_i,q_0b_1,\ldots,q_0b_m)$ is a syzygy of $I$.  Therefore, 
$$(s-\sum_{i=1}^m b_iq_i,q_0b_1,\ldots,q_0b_m)^t=t_1C_1+\ldots+t_rC_r,$$
for some $t_1,\ldots,t_r\in R$ where $C_1,\ldots,C_r$ denote the columns of $\psi$.
Multiplying out by the row matrix $[q_0\ldots q_m]$ yields
$$((q_0s-\sum_{i=1}^m b_iq_iq_0)+\sum_{i=1}^mb_iq_iq_0)=\sum_{i=1}^rt_ij_i.$$Thence $q_0s\in \mathcal{J}$. Since $q_0$ is a unit,  $s\in \mathcal{J}$ as desired.
\end{proof}
%%================================
We now return to the setting of the main Theorem \ref{TDegree} with the notation of Corollary~\ref{Cmaxrank}. For any point $q\in Y$, the fiber $\Phi^{-1}(q)$ is determined by the ideal $\fa=(I_q:_RI^{\infty})$. In particular, $\deg(\Phi)=e(R/\fa)$ for a general choice of $q$, see for example  \cite[Proposition 3.6]{KPU}. 
% According to  Lemma \ref{Lrow}, $\fa=\mathcal{J}:_RI^{\infty}$ where $\mathcal{J}:=I_1(q\cdot \psi)$. We also notice that $\mathcal{J}$ is generated by elements of degrees $d_1,\ldots, d_r$, according to the notations of Corollary \ref{Cmaxrank}. 

 The next proposition has already been obtained as a corollary of Theorem~ \ref{TDegree} under weaker conditions. Here, we give an alternative proof  using the idea of row ideals.
%======================================================================
\begin{Proposition}\label{Prowideal} Let $\Phi:X\dasharrow Y$ be a generically finite dominant rational map and let 
	$$\bigoplus^rR(-d-d_i)\xrightarrow{\psi} \bigoplus^{m+1}R(-d)\to I\to 0$$ 
	stand for the minimal  graded presentation of the base ideal $I=(f_0,\ldots,f_m)$ of $\Phi$.
	
If the  base locus of $\Phi$ has dimension at most zero and $R$ is Cohen--Macaulay,  then $\Deg(\Phi)\leq e(R)d_1\cdots d_{t-1}d_r$ where $t=\dim X$.
\end{Proposition}
\begin{proof} Without loss of generality, we may assume that the base field is infinite.  Pick  a general point $q=(q_0,\ldots,q_m)\in Y$  and  let $I_q$ be as above. 
Write $I_q=Q_1\cap\ldots \cap Q_l\cap Q_{l+1}\cap \ldots\cap Q_s$ for the primary decomposition of $I_q$ where $\sqrt Q_i\not \supset I$ for $i\leq l$ and $\sqrt Q_i \supset I$ for $i>l$.  Clearly, then 
 $$(I_q:_RI)=(I_q:_RI^{\infty})\cap (Q_{l+1}:_RI)\cap \ldots\cap (Q_s:_RI).$$ 
 Since the map is generically finite, the remark preceding Proposition \ref{Prowideal} gives $\codim(I_q:I^{\infty})=t$. Also,  $\codim(Q_{i}:I)\geq \codim(I)\geq \dim X=t$.  Thus, $\codim(I_q:I^{\infty})=\codim(I_q:I)=t$. Now let  $\mathcal{J}:=I_1(q\cdot \psi)$. Note that $\mathcal{J}$ is generated by elements of degrees $d_1,\ldots, d_r$. By Lemma~\ref{Lrow}, $\mathcal{J}=(I_q:I)$. Since $R$ is Cohen-Macaulay, $\mathcal{J}$ contains a regular sequence,  say $\A$, of length $t$. According to Lemma~\ref{lcrucial}, $\A$ can be chosen in degrees  $d_1,\ldots,d_{t-1},d_r$. The result now follows from the fact that 
 $$e(R/(I_q:_RI^{\infty}))\leq e(R/\mathcal{J})\leq e(R/\A)=e(R)d_1\cdots d_{t-1}d_r.$$ 
\end{proof}

Our examples show that without the standing assumptions in Proposition \ref{Prowideal} the last inequalities in the above proof do not hold. Although the inequality  $\Deg(\Phi)\leq e(R)d_1\cdots d_{t-1}d_r,$   still holds by Corollary \ref{Cmaxrank}.

\medskip

{\sc Acknowledgments.}
The first and second authors thank the Franco-Brazilian Network in Mathematics for providing facilities to mutual interchange visits.

%%%%%%%%%%%%%%%%%%%%%%%%%%%%%%%%%%%%%%%%%%%%%%%%%%%%%%%%%%%%%%%%%%%%%%%%%%%%%%%%%%%%%%

\noindent {\bf Addresses:}

\medskip

\noindent  {\sc Marc Chardin}\\
Institut de Math\'ematiques de Jussieu, CNRS \& Sorbonne Universit\'e\\ 
4 Place Jussieu 75005 Paris, France\\
{\em e-mail}: marc.chardin@imj-prg.fr

\medskip

\noindent {\sc Seyed Hamid Hassanzadeh}\\
Centro de Tecnologia - Bloco C, Sala ABC\\
 Cidade Universit\'{a}ria da Universidade Federal do Rio de Janeiro,\\
21941-909  Rio de Janeiro, RJ, Brazil\\
{\em e-mail}: hamid@im.ufrj.br 

\medskip

\noindent {\sc Aron Simis}\\
Departamento de Matem\'atica, CCEN\\ 
Universidade Federal de Pernambuco\\ 
50740-560 Recife, PE, Brazil\\
{\em e-mail}:  aron@dmat.ufpe.br

\end{document}